\documentclass{amsart}

\usepackage{color}
\usepackage[
frame,cmtip,arrow,matrix,line,graph]{xy}
\definecolor{darkgreen}{RGB}{47,139,79}
\definecolor{darkblue}{RGB}{36,24,130}

\usepackage{hyperref}
\hypersetup{
    colorlinks=true, 
    linktoc=all,     
    linktocpage,
    citecolor=darkgreen,
    filecolor=black,
    linkcolor=darkblue,
    urlcolor=black
}

\usepackage{amsmath}
\usepackage{amscd}
\usepackage{amssymb}
\usepackage{graphicx}
\usepackage[all]{xypic}

\usepackage{tikz}
\usetikzlibrary{calc}
\usetikzlibrary{patterns}
\usetikzlibrary{matrix}
\usepgflibrary{arrows}
\usetikzlibrary{arrows}
\usetikzlibrary{decorations.pathreplacing}
\usetikzlibrary{decorations.pathmorphing}
\usepackage{enumerate}
\usepackage{verbatim}

\numberwithin{equation}{section}

\theoremstyle{plain} 
\newtheorem{theorem}[equation]{Theorem}
\newtheorem{lemma}[equation]{Lemma}

\newtheorem{proposition}[equation]{Proposition}

\newtheorem{Th}{Theorem}

\theoremstyle{definition}
\newtheorem{definition}[equation]{Definition}

\newtheorem{examples}[equation]{Examples}

\newtheorem{remark}[equation]{Remark}
\newtheorem*{remark*}{Remark}

\newcommand{\del}{\partial}

\def\doCal#1{%
\ifx#1\doAllCalEnd\def\doAllCal{\relax}\else%
 \expandafter\edef\csname#1cal\endcsname{{\noexpand\mathcal #1}}\fi}
\def\doAllCal#1{\doCal#1\doAllCal}
\doAllCal ABCDEFGHIJHLMNOPQRSTUVWXYZabcdefghijklmnopqrstuvwxyz\doAllCalEnd

\def\doBar#1{%
\ifx#1\doAllBarEnd\def\doAllBar{\relax}\else%
 \expandafter\edef\csname#1bar\endcsname{{\noexpand\overline{#1}}}\fi}
\def\doAllBar#1{\doBar#1\doAllBar}
\doAllBar ABCDEFGHIJHLMNOPQRSTUVWXYZabcdefghijhlmnopqrstuvwxyz\doAllBarEnd

\def\doWiggle#1{%
\ifx#1\doAllWiggleEnd\def\doAllWiggle{\relax}\else%
 \expandafter\edef\csname#1wiggle\endcsname{{\noexpand\tilde{#1}}}\fi}
\def\doAllWiggle#1{\doWiggle#1\doAllWiggle}
\doAllWiggle ABCDEFGHIJHLMNOPQRSTUVWXYZabcdefghijklmnopqrstuvwxyz\doAllWiggleEnd

\DeclareMathOperator{\Aut}{Aut}

\DeclareMathOperator{\GL}{GL}
\DeclareMathOperator{\Hom}{Hom}

\DeclareMathOperator{\Link}{Link}

\def\Math#1{\def\MathString{#1}\futurelet\MathDelim\MathChoose}
\def\MathChoose{\ifmmode\let\MathDo\MathString%
              \else\let\MathDo\MathSkip\fi%
              \MathDo}
\def\MathSkip{\ifx\MathDelim/\def\MathDo{$\MathString$\EatOne}%
              \else\def\MathDo{$\MathString$}\fi%
              \MathDo}

\def\Text#1{\def\TextString{#1}\futurelet\TextDelim\TextSkip}
\def\TextSkip{\ifx\TextDelim/\def\TextDo{\TextString\EatOne}%
              \else\let\TextDo\TextString\fi%
              \TextDo}
\def\EatOne#1{}

\def\SkipToEndScan#1\EndScan{}
\def\Scan#1#2#3{\ifx#1#2#3\expandafter\SkipToEndScan\fi\Scan#1}
\def\Upper#1{%
\Scan#1aAbBcCdDeEfFgGhHiIjJkKlLmMnNoOpPqQrRsStTuUvVwWxXyYzZ#1#1\EndScan}
\def\Phrase#1 #2/#3/#4=#5 #6/#7/#8.{%
\expandafter\edef\csname#2#3\endcsname{\noexpand\Text{#6#7}}
\expandafter\edef\csname\Upper#2#3\endcsname{\noexpand\Text{\Upper#6#7}}
\expandafter\edef\csname#1#2#3\endcsname{\noexpand\Text{#5 #6#7}}
\expandafter\edef\csname\Upper#1#2#3\endcsname{\noexpand\Text{\Upper#5 #6#7}}
\expandafter\edef\csname#2#4\endcsname{\noexpand\Text{#6#8}}
\expandafter\edef\csname\Upper#2#4\endcsname{\noexpand\Text{\Upper#6#8}}
}

\newcommand{\pdash}{$p$\kern1.3pt-}

\Phrase a bad//s=a bad//s.

\newcommand{\double}[1]{D^2(#1)}

\newcommand{\Rdouble}[1]{D^r(#1)}

\newcommand{\decouple}[1]{\delta (#1)}

\newcommand{\s}{\sigma}
\newcommand{\PP}{\mathcal{P}}
\newcommand{\rar}{\longrightarrow}
\newcommand{\op}{\oplus}

\begin{document}

\title{The double of a simplicial complex}

\author{Kathryn Lesh}
\author{Bridget Schreiner}
\author{Nathalie Wahl}

\date{\today}

\maketitle

\begin{abstract}
  We introduce the notion of doubling and $r$-tupling for simplicial complexes, a notion reminiscent of matching complexes in graph theory. 
We prove a connectivity result for such complexes and relate $r$-tupling to stabilizing $r$ times faster in homological stability. 
  \end{abstract}

\section{Introduction}

Given a simplicial complex~$X$, one can define a new simplicial complex, its {\em double} $\double{X}$, whose vertices are the edges of $X$, and where a collection of $p+1$ edges form a $p$-simplex in $\double{X}$ if they are disjoint, and together form a $2p+1$-simplex of $X$. One can likewise define the $r$-tupling $\Rdouble{X}$ of $X$ for any natural number, where the vertices of $\Rdouble{X}$ are now the $(r-1)$-simplices of $X$, see Definition~\ref{def:rdouble}. (When $r=1$, the $r$-tupling is $X$ itself.)

Recall that a simplicial complex is called weakly Cohen-Macaulay (wCM for short) of dimension $n$ if it is $(n-1)$--connected, and the link of any $p$--simplex is $(n-p-2)$--connected.
The goal of this note is to prove the following result: 

\begin{Th}\label{thm:kdouble}
Suppose $X$ is wCM of dimension $n$. Then its $r$-tupling $\Rdouble{X}$ is wCM of dimension $\left\lfloor \frac{n+1-r}{r+1}\right\rfloor$. 
\end{Th}

The doubling of a simplicial complex is closely related to the notion of a {\em matching complex} in graph theory. Recall that the matching complex of a graph $\Gamma$ is the simplicial complex whose vertices are the edges of $\Gamma$, and higher simplices the disjoint collections; a top simplex in that simplicial complex is a complete matching of the vertices in the graph. 
When $X=\Delta^n$ is the $n$-simplex,  the double $\double{\Delta^n}$ identifies with the matching complex of the complete graph on $n+1$ vertices. More generally, the $r$-tupling $\Rdouble{\Delta^n}$ identifies with what is known as the matching complex of the complete $r$-hypergraph on $n+1$ vertices, see e.g.~\cite{Athanasiadis}. 

The connectivity properties of matching complexes of the complete graph and $r$-hypergraphs  have been studied by several authors, and 
the above result can be seen as a generalization of \cite[Cor 4.2]{BLVZ} (when $r=2$) and \cite[Thm~1.2]{Athanasiadis} (for all $r$'s) that give the case $X=\Delta^n$ in our theorem via the identification just explained. Our proof in fact relies on these earlier results. When $r=2$, the result is known to be sharp for $\Delta^n$. Indeed, the homology of $D^2(\Delta^n)$ is non-zero (most often 3-torsion) in degree $\lfloor\frac{n-1}{3}\rfloor$ for all $n$ except 7,8,10 and possibly 13, see  \cite[Thm~1.3, Thm~1.6]{ShaWac}, building on \cite{Bouc}. 
A better connectivity bound in the case of the simplex when $r>2$ would not directly improve our result since step 1 of our proof uses this exact form of bound, namely that for $k$ a natural number,  $k\le \left\lfloor \frac{n+1-r}{r+1}\right\rfloor -1$  is equivalent to $k\le n-(k+2)r$.  
(When $r=1$, Theorem~\ref{thm:kdouble} gives a bound half as good as the original bound for $X=D^1(X)$, but of course the proof, and in particular step 1 of the proof, takes an unnecessary detour in that case.)

\subsection*{Relationship to homological stability}

Our study of $r$-tupling of simplicial complexes arose in homological stability, in questions related to {\em destabilization complexes}. Recall that 
a sequence of groups $$G_1\to \dots \to G_n\to G_{n+1}\to \dots$$ satisfies homological stability if the map $H_i(G_n)\to H_i(G_{n+1})$ is an isomorphism when $n$ is large enough.
Stability holds for example for $G_n=\GL_n(R)$ for any ring~$R$ with finite stable rank,  $G_n=\Sigma_n$ the symmetric group on $n$ letters, or $G_n=\Aut(F_n)$ the automorphism group of the free group of rank~$n$, see e.g.,~\cite{RW-Wahl} or \cite{WahlICM} for further examples, references and an introduction to the subject.

Homological stability results for families of groups are typically proved using the action of the groups on certain simplicial spaces, or simplicial complexes.
As shown in \cite{RW-Wahl}, when the groups $\{G_n\}_{n\ge 1}$ together form a braided monoidal groupoid, as is most often the case in examples, there is a family of semi-simplicial sets $W_n$, called the destabilization complexes, whose connectivity properties govern homological stability in the sense that,  if $W_n$ is $\left(\frac{n-a}{k}\right)$--connected for all~$n$ and some $k\ge 2$, then the map
$H_i(G_n)\to H_i(G_{n+1})$ is an isomorphism for all $i\le \frac{n-a+1}{k}$. See Theorems~A and~3.1 in \cite{RW-Wahl}.

Conjecture C in the same paper states that a form of converse of this result should hold, in the sense that if homological stability holds in such a fashion, then, at least homologically, the destabilization complexes $W_n$ should be highly connected. As we will see, our Theorem~\ref{thm:kdouble} can be interpreted as proving this conjecture in the special case of ``fast stabilization'' when the groups together form a symmetric monoidal groupoid.

We will use the following observation: if a sequence $G_1\to G_2\to G_3\to \dots$ satisfies homological stability, so does the sequence $G_2\to G_4\to G_6\to \dots$, or more generally $G_r\to G_{2r}\to G_{3r}\to \dots$ for any $r\ge 1$.
The above conjecture, if true, would imply the following: whenever the destabilization complex $W_n$ associated to the standard ``+1" stabilization is highly connected, so should the destabilization complex $W_n^r$ associated to the ``+$r$'' stabilization for any $r$. A consequence of Theorem~\ref{thm:kdouble} is that, under mild assumptions,  this is indeed the case. See Proposition~\ref{prop:Sr} for a precise statement. To be able to apply our main result, we will use that, in good conditions,  the connectivity of the semi-simplicial sets $W_n$ and $W_n^r$ is determined by that of their underlying simplicial complexes $S_n$ and $S_n^r$, see Definition~\ref{def:WS} and \cite[Thm 2.10]{RW-Wahl}.   

Note that, as detailed in Remark~\ref{rem:ranges}, the connectivity bound for $W^r_n$ we obtain this way does not give optimal stability results in general.

\medskip

The complexes $W_n^r$ and $S_n^r$ come with an action of the group $G_{nr}$. Hence their first non-trivial homology defines a family of representations of the groups. In the particular case of the symmetric groups, where $S_n=\Delta^n$, Bouc identifies the first non-trivial homology group of the double $D^2(S_n)$ as symmetric groups representations, describing it as a specific direct sums of Specht modules, see \cite[Sec 3.2]{Bouc}.

\medskip

A variant of the $r$-tupling construction has arisen in current work by Belmont, Quigley, and Vogeli, who study an equivariant version of homological stability in the formalism of \cite{RW-Wahl}. In their case, for a group $G$, the $G$-equivariant analogue $W_n(G)$ of the destabilization complex $W_n$ has $p$-simplices given by ordered $(k_0 + ... + k_p)$-tuples of distinct vertices in the non-equivariant $W_n$, where $k_i = |G/H_i|$ for some subgroup $H_i \leq G$. They show, using a similar argument to the one given here, that high connectivity of the complexes $W_n$ implies high connectivity of the complexes $W_n(G)$.

Using the $E_k$--cells formalism of \cite{GKRW}, homological stability can be proved using the splitting complex instead of the destabilization complex. Connectivity properties of the destabilization complex and splitting complex are essentially equivalent by \cite[Prop 7.1]{RWcells}.
It should thus be possible to formulate and prove a version of our main result aimed at splitting complexes.

\medskip

\subsection*{Organization of the paper:} We define the doubling and $r$-tupling of a simplicial complex in Section~\ref{sec:matching}, and relate that notion to that of matching complexes. 
The proof of Theorem~\ref{thm:kdouble} is given in Section~\ref{sec:proof}. Finally Section~\ref{sec:stability} details the relationship to homological stability. 

\subsection*{Acknowledgments} This project started at the Women in
Topology workshop held at MSRI in 2017. We thank the organizers of the
program for setting this up and MSRI for hosting us for the week. We
thank Thomas Mikhail and the referees for helpful comments on the paper. We also thank the Newton Institute in Cambridge for support and hospitality both during the program Homotopy
Harnessing Higher Structures in 2018 and  Equivariant Homotopy Theory in Context in 2025, where work on this paper was undertaken (EPSRC grants EP/K032208/1, EP/R014604/1, and EP/Z000580/1).  The last author was also supported by the Danish National Research Foundation through the Copenhagen Centre for Geometry and Topology (DNRF151).

\section{Doubling, $r$-tupling, and matching complexes}\label{sec:matching}

Recall that a simplicial complex $X=(S,\PP)$ is the data of a set of vertices $S$ together with a collection $\PP$ of subsets of $S$ including all the singletons and closed under taking subsets. A  $p$-simplex of $X$ is then a subset $\tau=\{x_0,\dots,x_p\}\in \PP$ of cardinality $p+1$.

\begin{definition}\label{def:rdouble}
  The {\em $r$-tupling} of a simplicial complex $X$ is the simplicial complex $\Rdouble{X}$ with set of vertices the set of $(r-1)$--simplices of $X$ and $p$-simplices the collections $\{\tau_0,\dots,\tau_p\}$ with the property that  $\bigcup_{i=0}^p\tau_i$ is a $((p+1)r-1)$-simplex of $X$.
  When $r=2$, we call $\double X$ the {\em double} of $X$. 
\end{definition}

For example, the double $\double{\Delta^2}$ has three vertices and no edges, while  $\double{\Delta^3}$ has six vertices and three (disjoint) edges (see Figure~\ref{fig:DD}).

\begin{figure}
\includegraphics[width=0.6\textwidth]{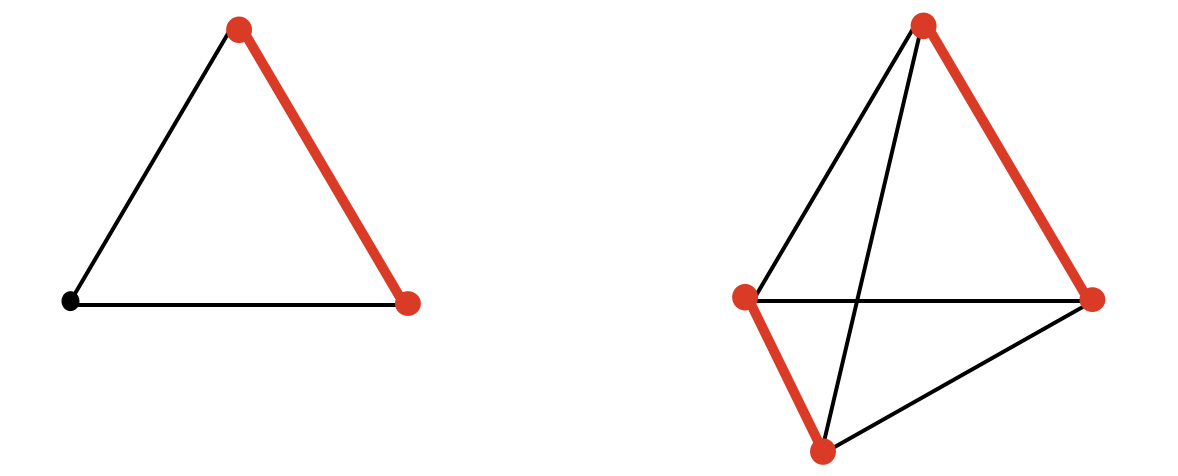}
\caption{Vertex of $D^2(\Delta^2)$ and edge of $D^2(\Delta^3)$}\label{fig:DD}
\end{figure}

\medskip

The {\em matching complex} of a graph $\Gamma$ is the simplicial complex $\mathcal{M}(\Gamma)$ with vertices the edges of $\Gamma$, and $p$--simplices the collections $\{e_0.\dots,e_p\}$ of pairwise disjoint edges. (See eg., \cite[Def 3.4]{BFMWZ}.)
Matching complexes were first introduced by Bouc \cite{Bouc} to study posets of subgroups. The poset of simplices of the double $\double{\Delta^n}$ identifies with the poset of elementary abelian $2$-subgroups of the symmetric group $\Sigma_{n+1}$ that are generated by transpositions.
Bouc began the  study of the connectivity properties of   $\double{\Delta^n}$ as well as the $\Sigma_{n+1}$-representation defined by its first non-trivial homology group.

Let $K_{n+1}$ be the complete graph on $n+1$ vertices. Note that there is an isomorphism $\double{\Delta^n}\cong \mathcal{M}(K_{n+1})$. Indeed, the 1-skeleton of $\Delta^n$ is the complete graph $K_{n+1}$, giving an isomorphism on the set of vertices. And higher simplices are defined in the same way in both cases, since any subset of vertices defines a simplex in $\Delta^n$.

In \cite{Athanasiadis}, Athanasiadis considers the more general {\em $r$--hypergraph matching complexes $\mathcal M_n(r)$}, where $\mathcal M_n(2)=\mathcal{M}(K_n)$ is the matching complex of the complete graph, and $\mathcal M_n(r)$ is defined more generally as the simplicial complex whose vertices are the subsets of $[n]=\{1,\dots,n\}$ of cardinality $r$, and simplices the collections of pairwise disjoint such. We have the more general identification $\Rdouble{\Delta^n}\cong \mathcal{M}_{n+1}(r)$. 

\medskip

Recall that a simplicial complex $X$ is called {\em Cohen-Macaulay} (abbreviated {\em CM}) {\em of dimension $n$} if it is of dimension $n$, and for $-1\le p\le n$, the link of any $p$-simplex in $X$ is $(n-p-2)$--connected. A $(-1)$-simplex being empty, the condition in that case is that $X$ itself must be (at least) $(n-1)$-connected. We say that $X$ is {\em weakly} Cohen-Macaulay (or {\em wCM}) of dimension $n$ if only the connectivity condition for the links is satisfied. This is equivalent to requiring that the $n$-skeleton of $X$ is CM of dimension $n$ since any wCM complex of dimension $n$ has actual dimension at least $n$ because the links of $(n-1)$--simplices in such a complex are required to be $(-1)$-connected, that is non-empty, and the connectivity condition is trivial for simplices of dimension $p\ge n$. 

\medskip

It was shown in \cite{BLVZ} that $\mathcal M_n(r)$ is wCM of dimension $\nu(r)=\left\lfloor\frac{n-2}{2r-1}\right\rfloor$, while in 
\cite{Athanasiadis}, it is shown that its $\mu_n(r)$--skeleton is shellable for
$$\mu_n(r)=\left\lfloor\frac{n-r}{r+1}\right\rfloor,$$
where {\em shellable} is a condition known to be stronger than CM, see e.g.~\cite[App.]{Bjorner}. 
When $r=2$, the two bounds agree, but for $r>2$, the second bound is better.

We deduce the following result: 

\begin{theorem}\label{thm:Delta}\cite{Athanasiadis}
$\Rdouble{\Delta^n}$ is wCM of dimension $\left\lfloor\displaystyle\frac{n+1-r}{r+1}\right\rfloor$. 
\end{theorem}

\begin{proof}
We have that  $\Rdouble{\Delta^n}$  is isomorphic to   $\mathcal M_{n+1}(r)$, which is wCM of dimension $\lfloor\frac{n+1-r}{r+1}\rfloor$, since its skeleton of that dimension is shellable by \cite[Thm~1.2]{Athanasiadis}, and shellable implies (homotopy) CM, see e.g.~\cite[App.]{Bjorner}.  
  \end{proof}

Note that the non-trivial homology groups of the complexes $\mathcal M_n(r)$ have been studied as symmetric groups representations, see e.g.,~\cite{Bouc}, \cite{Shareshian-Wachs} and \cite{Wachs}, also for the relationship to other complexes associated to groups.

\section{Deducing the connectivity of $\Rdouble{X}$ from that of $\Rdouble{\Delta^n}$}\label{sec:proof}

Our proof of Theorem~\ref{thm:kdouble} is inspired by the proof of Theorem~3.6 in
\cite{Hatcher-Wahl}, and we will use as input the known connectivity of $\Rdouble{\Delta^n}$.

For a simplicial complex $X$, let $sX$ denote its first barycentric subdivision, with set of vertices all the simplices of $X$ and where a collection of simplices $\{\s_0,\dots,\s_p\}$ forms a simplex in $sX$ is there is a chain of face inclusions $\s_0<\dots<\s_p$ in $X$. (The complex $sX$ can also be described as the geometric realization of the face poset of $X$.) 
As in \cite{Hatcher-Wahl}, we denote by $X_m$ the subcomplex of  $sX$ consisting of all the simplices of $X$ with at least $m$ vertices (i.e.~of dimension at least $m-1$). 
The idea of the proof of our main theorem is to use $X_{r(k+2)}$ as an intermediate space to show that the $r$-tupling of $X$ is $k$--connected, because locally, in $X_{r(k+2)}$, we always have a simplex of $X$ with at least $r(k+2)$ vertices, which corresponds to a simplex of $\Rdouble X$ of dimension at least $k+1$. And $(k+1)$--simplices is exactly what is needed to fill in a $k$-sphere with a disc.

We start with an enhancement of a connectivity result for $X_m$ proved in \cite{Hatcher-Wahl},  promoting ``connectivity" to ``wCM".
(When $X$ is CM, this lemma follows from the more general rank selection theorem for posets, see \cite[Thm 6.4]{Baclawski}.)

\begin{lemma}\label{lem:wCM}
Suppose $X$ is wCM of dimension $n$. Then $X_m$ is wCM of dimension $n-m+1$. 
\end{lemma}

\begin{proof} 
  We know from  \cite[Lem 3.8]{Hatcher-Wahl} that $X_m$ is $(n-m)$--connected, so we are left to show that the link of a  $p$-simplex in $X_m$ is at least $(n-m-p-1)$--connected for any $p\ge 0$.

  Let $\underline{\s}=\s_0<\dots<\s_p$ be a $p$--simplex of $X_m$, so that  $\s_0$ is a simplex of $X$ of dimension at least an $(m-1)$ and $\s_p$ is a $q$-simplex of $X$, for some $q\ge p+m-1$. 
  The link of $\underline \s$ is the join $L_D *L_M* L_U$ with
  \begin{itemize}
  \item[-] the ``down-link'' $L_D$: the subposet of $X_m$ of simplices $\tau$ that are faces of $\s_0$,
  \item[-]   the ``middle-link'' $L_M$: the subposet of simplices $\tau$ with $\s_i<\tau<\s_{i+1}$, 
 \item[-]  the ``up-link'' $L_U$: the subposet of simplices $\tau$ that contain $\s_p$ as a face. 
\end{itemize}
 
  Write $m_i$ for the number of vertices of $\s_i$. Then $m\le m_0<\dots<m_p$. 

The down-link $L_D$ identifies with $(\del \Delta^{m_0-1})_m$, which is $(m_0-m-2)$--connected as $\del \Delta^{m_0-1}$ is CM of dimension $m_0-2$. 

The middle-link $L_M=L_1*\dots*L_p$ splits as the join of the sublinks $L_i$ of simplices $\tau$ with $\s_{i-1}<\tau<\s_{i}$. Now we can write $\s_{i}=\s_{i-1}*\Delta^{m_{i}-m_{i-1}-1}$ and $L_i$ identifies with
$\del \Delta^{m_{i}-m_{i-1}-1}$, as we can add any strict face of this extra simplex to $\s_{i-1}$. So $L_i$  is $(m_{i}-m_{i-1}-3)$--connected.

Finally the up-link $L_U$  identifies with the first barycentric subdivision of the link of $\s_p$ in~$X$, since larger simplices are automatically large enough, and any larger simplex is obtained by joining a simplex from the link.  It is hence $(n-m_p-1)$--connected by our assumption on $X$. Hence the whole link has connectivity at least
$$(m_0-m)+\sum_{i=1}^p(m_i-m_{i-1}-1)+(n-m_p+1)-2=n-m-p-1$$
as needed. 
\end{proof}

The wCM property will allow us to use \cite[Thm~2.4]{Galatius-RW}, which tells us that when we use the connectivity of $X_m$ to extend a map $f:S^k\to X_m$ from sphere $S^k$ to a disc $D^{k+1}$, we can do this in such a way that the map is simplexwise injective on the new simplices.

\medskip

For $\s$ a $p$-simplex of $\Rdouble{X}$, we will denote $\delta(\s)$ the underlying $(r(p+1)-1)$-simplex of~$X$.

\begin{lemma}    \label{lem:doublelink}
For $\tau$ a $p$-simplex of $\Rdouble{X}$, 
\[
\Link_{\Rdouble{X}}\tau\cong\Rdouble{\Link_{X} \decouple{\tau}}. 
\]
\end{lemma}

\begin{proof}
 By definition, a simplex $\nu$ is in $\Link_{\Rdouble{X}}\tau$ if and only if $\nu*\tau$ is a simplex of $\Rdouble X$, which happens, again by definition, if and only if
  $\delta(\nu)*\delta(\tau)$ is a simplex of $X$. This can be reformulated as saying that $\delta(\nu)$ is in $\Link_{X} \decouple{\tau}$, which happens if and only if $\nu\in \Rdouble{\Link_{X} \decouple{\tau}}$. 
  \end{proof}

Note also that, in any simplicial complex $X$, if $\s$ is a simplex of $X$ and $\tau\in \Link_X(\s)$, then $$\Link_{\Link(\s)}(\tau)\cong \Link_X(\s*\tau).$$ 
It follows from this that, when $X$ is wCM of dimension $n$, then for any $p$-simplex $\s$, the link $\Link_X(\s)$ is wCM of dimension $n-p-1$.

\begin{proof}[Proof of Theorem~\ref{thm:kdouble}]
We will show that 
\begin{equation*}
\Rdouble X \textrm{ is } \Big\lfloor\frac{n-2r}{r+1}\Big\rfloor=(\Big\lfloor\frac{n+1-r}{r+1}\Big\rfloor-1) \textrm{-connected for any wCM complex $X$ of dimension $n$.}
\end{equation*}
This connectivity result gives the desired wCM property since, by Lemma~\ref{lem:doublelink}, if $\tau$ is a $p$-simplex of $\Rdouble{X}$, then 
$\Link_{\Rdouble{X}}\tau\cong\Rdouble{\Link_{X} \decouple{\tau}}$. Now if $X$ is wCM of dimension $n$, then  $\Link_{X} \decouple{\tau}$ is wCM of dimension $n-(r(p+1)-1)-1=n-r(p+1)$. 
Hence we can apply the connectivity result to  $\Link_{X} \decouple{\tau}$, which gives that the link is at least $\Big\lfloor\frac{n-r(p+1)-2r}{r+1}\Big\rfloor$-connected. Finally we check that
$$\Big\lfloor\frac{n-r(p+1)-2r}{r+1}\Big\rfloor=\Big\lfloor\frac{n-rp-3r}{r+1}\Big\rfloor\ge \Big\lfloor\frac{n-rp-3r-1-p}{r+1}\Big\rfloor=\Big\lfloor\frac{n-r+1}{r+1}\Big\rfloor-p-2$$
for any $p\ge -1$, giving the required connectivity condition on links. 

\medskip
  
So fix a simplicial complex $X$ that is wCM of dimension $n$ and let $k\le \big\lfloor\frac{n-2r}{r+1}\big\rfloor\le \frac{n-2r}{r+1}$. Then $(r+1)k\le n-2r$, or equivalently, $k\le n-r(k+2)$. 

Let $f:S^k\to \Rdouble{X}$ be a map, which we assume simplical for some triangulation $K$ of~$S^k$. The plan of the proof is as follows: 
\begin{itemize}
\item[step 1] Use $f$ to construct a map $\hat f:K_1 \to X_{r(k+2)}$, for $K_1$ a subdivision of $K$. By   \cite[Lem 3.8]{Hatcher-Wahl}, we know that $\hat f$ extends to a map $g:D^{k+1}\to X_{r(k+2)}$, simplicial with respect to a triangulation $L_1$ of the disc restricting to $K_1$ on the boundary. We will choose $g$ to be simplexwise injective. 
\item[step 2] Use $g$ to construct a map $h:L_2\to \Rdouble{X}$, for $L_2$ a subdivision of $L_1$, with the property that $h|_{S^k}\simeq f$. 
\end{itemize}

We start with step 1. We will build $\hat f$ and $K_1$ inductively on the skeleta of the first barycentric subdivision $sK$ of $K$, with the following property: 
\begin{itemize}
\item[$(*)$] For each vertex $v$ of $K_1$ in the interior of a simplex $\s_0<\s_1<\dots<\s_p$ of $sK$, we have that $\delta(f(\s_0))\le \hat f(v)$. 
\end{itemize}
We start with the vertices of $sK$. Such a vertex corresponds to a simplex $\s$ of $K$. 
Because $X$ is wCM of dimension $n\ge k+r(k+2)$, we can, if the dimension of $\delta(f(\s))$ is less than $(r(k+2)-1)$, pick an $(r(k+2)-1)$--simplex $\hat f(\s)$ containing $\delta(f(\s))$ as a face, as any simplex in a wCM complex of dimension $n$ is the face of a simplex of dimension $n$. The injectivity condition is automatically satisfied on vertices. 

Now suppose that we have defined $\hat f$ on a subdivision of the $(m-1)$--skeleton of $sK$ simplexwise injective and satisfying $(*)$. Let $\underline\s=(\s_0<\dots<\s_m)$ be an $m$--simplex of $sK$. By induction, we have thus defined $\hat f$ on a subdivision of the boundary $\del\underline\s$, so that  the image of $\hat f$ on any vertex $v$ in the subdivided boundary always contains at least $\delta(f(\s_0))$ as a face, since such a $v$ lies in the interior of a face of $\underline \s$, whose minimal element $\s_i$ necessarily contains $\s_0$. Now $\s_0$ is a $p$--simplex of $K$ for some $p\le k$, with $f(\s_0)$ a $p'$--simplex of $\Rdouble{X}$ for some $p'\le p$, so that $\delta(f(\s_0))$ is a simplex of $X$ with $r(p'+1)$ vertices for some $p'\le p\le k$. Note that
$r(p'+1)\le r(k+1)<r(k+2)$. We can thus consider the restriction of $\hat f$ to  $\del\underline\s$ as a map
$$\hat f|_{\del \underline\s}=\delta(f(\s_0))*\hat f_1:  \del\underline\s\simeq S^{m-1}\ \longrightarrow\ X_{r(k+2)}$$
with $\hat f_1: \del \underline\s\to \Link_X(\delta(f(\s_0)))_{r(k+2)-r(p'+1)}$. 
Since $\Link_X(\delta(f(\s_0)))$ is wCM of dimension $n-r(p'+1)$, the target of $\hat f_1$ is 
wCM of dimension 
$$n-r(p'+1)-(r(k+2)-r(p'+1))+1=n-r(k+2)+1$$
by Lemma~\ref{lem:wCM}, so we can extend $\hat f_1$ to a map  $\hat f_2: D^m\to \Link(\delta(f(\s_0))) _{r(k+2)-r(p'+1)}$ in a simplexwise injective way for a triangulation $T$ of $D^m$ extending that of the boundary, as long as $m-1\le n-r(k+2)$, which holds since $m-1\le k-1\le n-r(k+2)-1$ by assumption. 
Now define $\hat f=\delta(f(\s_0)) * \hat f_2$ in the interior of the $m$-simplex $\underline \s$, now triangulated using the triangulation $T$ of $D^m\cong \Delta^m$ used by $\hat f_2$. 
This satisfies condition $(*)$ by construction, is still simplexwise injective, and hence gives the induction step, once this is done for all the $m$ simplices of~$sK$. 

\smallskip

By assumption $k\le n-r(k+2)$, so that \cite[Lem 3.8]{Hatcher-Wahl} gives a null-homotopy  $g:D^{k+1}\to X_{r(k+2)}$ of $\hat f$, simplicial with respect to a triangulation~$L_1$, agreeing with $K_1$ in the boundary of the disc. By Lemma~\ref{lem:wCM} and using \cite[Thm~2.4]{Galatius-RW}, we can assume that $g$ is simplexwise injective.

For step 2, we want to use $g$ to construct a map $\hat g:D^{k+1}\to \Rdouble{X}$, with $D^{k+1}$ now triangulated by a subdivision $L_2$ of~$L_1$.  
We build $\hat g$ inductively on the skeleta of $L_1$, satisfying that 
\begin{itemize}
\item[$(**)$] For each vertex $v$ of $L_2$ in the interior of a simplex $\tau$ of $L_1$ with $g(\tau)=\alpha_0<\dots<\alpha_p$ in $X_{r(k+2)}$, we have that $\hat g(v)$ is an $(r-1)$-face of $\alpha_p$. And if $\tau$ is in the boundary, lying in the interior of the simplex $\s_0<\dots<\s_j$ of $sK$, we require more specifically that $\hat g(v)=f(w)$ for $w$ a vertex of~$\s_0$.  
\end{itemize}
Note that the condition on the boundary is compatible with the condition on all simplices by condition $(*)$: if $v$ is in the interior of a simplex $\tau$ of $L_1$ and  in the interior of the simplex $\s_0<\dots<\s_j$ of $sK$, then the vertices of $\tau$ are necessarily interior to faces of $\s_0<\dots<\s_j$, and hence by $(*)$, each $\alpha_i$ contains at least $\delta(f(\s_0))$ as a face, so faces of  $\delta(f(\s_0))$ are also faces of $\alpha_p$.

To start the induction, let $v$ be a vertex of $L_1$. If $v\in K_1$, we set $\hat g(v)=f(w)$ with $w$ chosen as in $(**)$, and if not, we pick $\hat g(v)$ to be some $(r-1)$-face of the simplex $g(v)$, which is possible as $r(k+2) \ge r$. This satisfies $(**)$ by construction.

Note that this choice of $\hat g$ on the vertices of $K_1\subset L_1$ extends linearly to a map
$$\hat g|_{K_1}:K_1\rar \Rdouble{X}$$
since any simplex $\tau=\{v_0,\dots,v_j\}$ of $K_1$ lies inside a simplex  $\s_0<\dots<\s_j$ of $sK$, and $\hat g(v_i)=f(a_{s_i})$ for $a_{s_i}$ some vertex of $\s_j$. For the same reason, the map $\hat g|_{|K_1|}$ is homotopic to $f$, by linear interpolation.

\smallskip

Now suppose that we have defined $\hat g$ on a subdivision of the $(m-1)$--skeleton of $L_1$ satisfying $(**)$ and consider $\tau$ an $m$--simplex of $L_1$, not in the boundary, 
with $g(\tau)=\alpha_0<\dots<\alpha_m$ in $X_{r(k+2)}$, where we use that $g$ is simplexwise injective. 
By induction, we have defined $\hat g$ on its boundary so that the value of $\hat g$ on every vertex of $\partial \tau$ is an $(r-1)$-face of $\alpha_m$. Hence we have 
$$\hat g|_{\partial \tau}: \del \tau\simeq S^{m-1}\to \Rdouble{\alpha_m}.$$
Now $\alpha_m$ is a simplex of $X$ of dimension at least $r(k+2)+m-1$ since $\alpha_0$ has at least $r(k+2)$ vertices.

We have that $\Rdouble{\alpha_m}$ is $\left(\big\lfloor\frac{r(k+2)+m-r}{r+1}\big\rfloor-1\right)$--connected by Theorem~\ref{thm:Delta}. Now $m\le k+1$, so
$$\Big\lfloor\frac{r(k+2)+m-r}{r+1}\Big\rfloor-1\ge \Big\lfloor\frac{rm+r+m-r}{r+1}\Big\rfloor-1=m-1,$$ as needed.
So we can extend $\hat g$ in the interior of $\tau$, satisfying $(**)$ by construction. This finishes the proof.
\end{proof}

\section{Doubling and destabilization complexes}\label{sec:stability}

The $r$-tupling of a simplicial complex arises in the context of homological stability, when ``stabilizing fast'' in the way described in the introduction. We make this precise here. We start by recalling the destabilization complexes of \cite{RW-Wahl}. We will adopt for this section the categorical language of that paper, where groups are considered as automorphism groups in appropriate categories, and stabilization maps are induced by a monoidal structure in the category. 

\medskip

Let  $(\Ccal,\oplus,0)$ be a monoidal category, and $a,x$ objects in $\Ccal$. We will consider groups of the form $G_n=\Aut(a\op x^{\op n})$ with stabilization maps of the form
\begin{equation}\label{eq:stab}
\s_1:   G_n=\Aut(a\oplus x^{\op n})\ \xrightarrow{\op x}\ G_{n+1}=\Aut(a\oplus x^{\op n+1})
  \end{equation}
  adding the identity on the added $x$-summand.
  For example, the symmetric group $\Sigma_n$ is the automorphism group of the set $[n]$ with $n$ elements in the category of finite sets and injections, with disjoint union as monoidal structure, and the group $\GL_n(R)$ is the automorphism group of $R^{\oplus n}$ in an appropriate category of $R$-modules. 
  We refer to \cite[Sec 5]{RW-Wahl} for further examples, 
  and \cite[sec 1.1]{RW-Wahl} for how to construct an appropriate category $\Ccal$ from the groups $\{G_n\}_{n\ge 1}$, a ``sum'' $G_n\times G_m\to G_{n+m}$ and a compatible collection of morphisms  $B_{n} \to G_{n}$ from the braid groups.

 \medskip

  When $0$ is initial in $\Ccal$, we can associate to the pair of objects $(a,x)$ a semi-simplical set, and underlying simplicial complex: 
\begin{definition}\cite[Defn 2.1 and 2.8]{RW-Wahl}\label{def:WS}
Let $(\Ccal,\oplus,0)$ be a monoidal category with $0$ initial and $(a,x)$ a pair of objects in~$\Ccal$.
\begin{enumerate}
\item Define $W_{n}(a,x)_{\bullet}$ as the semi-simplicial set with $p$-simplices  $$W_n(a,x)_p=\Hom_{\Ccal}\left(x^{\oplus p+1},a\oplus x^{\oplus n}\right),$$ and face maps $d_i$ precomposing with $x^{\op i}\op \iota_x\op x^{\op p-i}$ for $\iota_x:0\to x$ the unique map.
\item Define $S_{n}(a,x)$ to be the simplicial complex with set of vertices $W_{n}(a,x)_{0}$ and such that $\{f_0,\dots,f_p\}$ forms a $p$-simplex if there exists a $p$-simplex $f\in W_n(a,x)_p$ with
  set of vertices $\{f_0,\dots,f_p\}$. 
\end{enumerate}
\end{definition}
In most examples arising in homological stability, the category $\Ccal$ is {\em locally homogeneous at $(a,x)$}, which can be formulated as saying  that the stabilization maps \eqref{eq:stab} 
are injective, and the destabilization complex has the form 
$$W_n(a,x)_p\cong G_n/G_{n-p-1}$$
where $G_{n-p-1}$ identifies with the stabilizer of the canonical $p$-simplex  $$\iota_{a\op x^{\op n-p-1}}\op x^{\op p+1}: x^{\op p+1}\cong 0\op x^{\op p+1}\rar a\op x^{\op n}.$$

A $p$-simplex in a semi-simplicial set has $p+1$ ordered vertices, obtained from applying repeated face maps. We will only consider here semi-simplicial sets with the property that the vertices of any simplex are all distinct, and the ordered collection of vertices determines the simplex, like an ordered version of a simplicial complex. For $W_n(a,x)$, this happens when  $\Ccal$ is {\em locally standard at $(a,x)$} in the language of \cite{RW-Wahl}, see Proposition 2.6 of that paper. 

For a simplicial complex $X$, denote $X^{ord}$ the semi-simplicial set with a $p$-simplex for every $p$-simplex of $X$ and every choice of ordering of its vertices. 
When $(\Ccal,\op,0)$ is symmetric monoidal, locally homogeneous and locally standard at $(a,x)$, 
Proposition 2.9 of \cite{RW-Wahl} gives an identification 
$$W_n(a,x)\cong S_n(a,x)^{ord}.$$
This is the situation we will consider in this section. (We expect that it is also possible to treat the {\em braided case}, where $W_n(a,x)\cong S_n(a,x)$, but it requires a separate argument.)

\begin{examples}\label{ex:FI} Let $(\Ccal,\op,0)=(\textrm{FI},\sqcup,\emptyset)$ be the category of finite sets and injections, with disjoint union as monoidal structure, $a=\emptyset$ and $x=\{*\}$. 
Then 
\begin{enumerate}
\item $W_n(a,x)$ is the semi-simplicial set of {\em injective words} in $[n]=\{1,\dots,n\}$, with $p$-simplices
  $$W_n(a,x)_p=\textrm{Hom}_{\textrm{FI}}([p+1],[n])=\{(a_0,\dots,a_p)\in [n]^{p+1}\ |\ a_i\neq a_j \textrm{ when } i\neq j\}.$$ 
\item $S_n(a,x)$ is isomorphic to the $(n-1)$-simplex $\Delta^{n-1}$. 
\end{enumerate}
Then a $p$-simplex of $W_n(a,x)\cong S_n(a,x)^{ord}$ is exactly the data of a $p$-simplex of $S_n(A,X)$, that is a subset of $[n]$, together with a choice of ordering of the elements of this subset, where any ordering gives a valid, and distinct, simplex in $W_n(a,x)$. 
\end{examples}

The semi-simplicial set $W_n(a,x)$ is called the {\em destabilization complex}, and Theorem~A in \cite{RW-Wahl} says that, under the homogeneity assumption, its connectivity ensures homological stability for the maps~\eqref{eq:stab}. Moreover, the connectivity of the simplicial complexes $S_n(A,X)$ under our assumption determines that of the semi-simplicial sets $W_n(A,X)$, see \cite[Thm~2.10]{RW-Wahl}.

\smallskip

Given $(a,x)$ in $(\Ccal,\op,0)$ as above, we can also consider the stabilization maps
\begin{equation}\label{eq:rstab}
 \s_r:  \Aut(a\oplus x^{\op n})\ \xrightarrow{\op x^{\op r}}\ \Aut(a\oplus x^{\op n+r})
  \end{equation}
  that go $r$ times faster. Homological stability for this fast stabilization is ruled by the semi-simplicial sets $W_n(a,x^{\op r})$. 
  Because homological stability for the maps stabilizing one $x$ at a time implies homological stability for the maps stabilizing $r$ $x$'s at once, it is to be expected that the connectivity of the semi-simplicial sets $W_n(a,x)$ should imply a connectivity result for the semi-simplicial sets $W_n(a,x^{\op r})$. We will see that this is the case, under good conditions, by first using that the connectivity properties of $W_n(a,x)$ and $W_n(a,x^{\oplus r})$ are tightly related to those of $S_n(a,x)$ and $S_n(a,x^{\oplus r})$, and then relating the latter simplicial complexe to the $r$-tupling construction on the former. Understanding what happened to these complexes under fast stabilization was our motivation for proving Theorem~\ref{thm:kdouble}. 

  \smallskip

Recall from \cite[Def 3.2]{Hatcher-Wahl} that a simplicial complex $Y$ is a complete join over a simplicial complex $X$ if there is a surjective map of simplicial complexes $\pi:Y\to X$ such that $\langle y_0,\dots,y_p\rangle$ is a  $p$-simplex of $Y$ if and only if
$\langle \pi(y_0),\dots,\pi(y_p)\rangle$ is a  $p$-simplex of $X$. This last condition can be reformulated as saying that 
if $\s=\langle x_0,\dots,x_p\rangle$ is a $p$-simplex of $X$, then $\pi^{-1}(\s)$ is the join $\pi^{-1}(x_0)*\dots*\pi^{-1}(x_p)$, which explains the name {\em complete join complex}.

\begin{proposition}\label{prop:Sjoin} Let $(\Ccal,\op,0)$ be a symmetric monoidal category with $0$ initial, locally homogeneous and locally standard at $(a,x)$. Then 
$S_{n}(a,x^{\oplus r})$ is a complete join complex over $\Rdouble{S_{nr}(a,x)}$.
\end{proposition}

\begin{proof}
By definition, a vertex of $S_{n}(a,x^{\oplus r})$ is a morphism $x^{\op r} \to a\op (x^{\op r})^{\op n}$ in $\Ccal$, which can be interpreted as an $(r-1)$-simplex of $W_{nr}(a,x)$. On the other hand, a vertex of $\Rdouble{S_{nr}(a,x)}$ is an $(r-1)$-simplex of $S_{nr}(a,x)$, that is a collection $\{f_1,\dots,f_r\}$ of morphisms $f_i: x \to a\op x^{\op nr}$ such that there exists $f\in W_{nr}(a,x)_{r-1}$ with vertices $(f_1,\dots,f_r)$, i.e. such that $f: x^{\op r} \to a\op x^{\op nr}$ restricts to $f_i$ on the $i$th factor. 
 In particular, we get a surjective map on vertex sets 
\[
\pi: \big(S_{n}(a,x^{\oplus r})\big)_0\cong W_{nr}(a,x)_{r-1}\ \longrightarrow\ \big(S_{nr}(a,x)\big)_{r-1}\cong \big(\Rdouble{S_{nr}(a,x)}\big)_0.
\]
taking a vertex of $S_{n}(a,x^{\oplus r})$, thought of as a simplex $\s$ of $W_{nr}(a,x)$, to its underlying simplex in  $S_{nr}(a,x)$, where we have used the notation $(X)_p$ for the set of $p$-simplices of a simplicial complex $X$.

Now a collection of vertices $\{g_0,\dots,g_p\}$ in $S_{n}(a,x^{\oplus r})$ defines a simplex if and only if there is a simplex $g: (x^{\op r})^{\op p+1} \to a\op (x^{\op r})^{\op n}$ in $W_{n}(a,x^{\oplus r})$ appropriately restricting to $g_0,\dots,g_p$, that is if there exists a permutation of the $g_i$'s such that $g=(g_{i_0},\dots,g_{i_p})$. Interpreting $g$ as an $(r(p+1)-1)$-simplex of $W_{nr}(a,x)$, it has vertices the collection $(g_{i_0}^1,\dots,g_{i_0}^{r},\dots,g_{i_p}^1,\dots,g_{i_p}^{r})$. 
Because $\Ccal$ is symmetric monoidal, locally homogenous and locally standard at $(a,x)$, 
this happens if and only if there exists an $(r(p+1)-1)$-simplex of $W_{nr}(a,x)$ with vertices any permutations of the $g_i^j$'s, or equivalently if the union of the vertices of the $g_i$'s forms a simplex in $S_{nr}(a,x)$. 
Hence  $\{g_0,\dots,g_p\}$ is a $p$-simplex of $S_{n}(a,x^{\oplus r})$ if and only if
$\{\pi(g_0),\dots,\pi(g_p)\}$ is a $p$-simplex of $\Rdouble{S_{nr}(a,x)}$, proving the claim. 
\end{proof}

To be able to use the above result to deduce a connectivity result for $S_n(a,x^{\op r})$ and $W_n(a,x^{\op r})$ from the connectivity result for $S_n(a,x)$, we need to know that the simplicial complexes $S_n(a,x)$ are wCM. This is a property that automatically holds in our situation if connectivity holds for all $n$ with a linear bound, see  \cite[Prop~2.9 and Cor 2.13]{RW-Wahl}.

\begin{proposition}\label{prop:Sr} Let $(\Ccal,\op,0)$ be a symmetric monoidal category with $0$ initial, locally homogeneous and locally standard at $(a,x)$. 
If $S_n(a,x)$ is $\big\lfloor\frac{n-c}{k}\big\rfloor$-connected for all $n$, then 
\begin{enumerate}
\item $S_{n}(a,x^{\oplus r})$ is wCM of dimension $\big\lfloor\frac{nr+k(2-r)-c}{k(r+1)}\big\rfloor$ for any $n,r\ge 1$. 
\item $W_{n}(a,x^{\oplus r})$ is $\big\lfloor\frac{nr+k(1-2r)-c}{k(r+1)}\big\rfloor$--connected for any $n,r\ge 1$.
\end{enumerate}
  \end{proposition}

  \begin{proof}
Using \cite[Prop~2.9 and Cor 2.13]{RW-Wahl}, we get that $S_n(a,x)$ is wCM of dimension $\big\lfloor\frac{n-c}{k}\big\rfloor+1$. 
Theorem~\ref{thm:kdouble} then gives that  $\Rdouble{S_{nr}(a,x)}$ is wCM of dimension 
$$\Big\lfloor\frac{\big\lfloor\frac{nr+k-c}{k}\big\rfloor+1-r}{r+1}\Big\rfloor=\Big\lfloor\frac{\big\lfloor\frac{nr+2k-rk-c}{k}\big\rfloor}{r+1}\Big\rfloor=\Big\lfloor\frac{nr+k(2-r)-c}{k(r+1)}\Big\rfloor.$$
Combining   Proposition~\ref{prop:Sjoin} with \cite[Prop 3.5]{Hatcher-Wahl} then gives that  $S_{n}(a,x^{\oplus r})$ is wCM of the same dimension, proving (1). 

Part (2) then follows from \cite[Prop~2.9 and Thm 2.10]{RW-Wahl}, since $\Ccal$ locally homogeneous and locally standard at $(a,x)$ implies homogeneity and standardness at $(a,x^{\op n})$. Indeed, local homogeneity at $(a,x^{\op r})$ and the first local standardness condition follow directly from the definition, being weaker properties than their versions at $(a,x)$, see Definition 1.2 and 2.5 in \cite{RW-Wahl}. For the second local standardness condition, this follows from the same condition at $(a,x)$ using \cite[Prop 2.6]{RW-Wahl}, as the result implies that $\iota_a\op x^{\op r}\op \iota_{x^{\op r}}$ and $\iota_a\op \iota_{x^{\op r}}\op x^{\op r}$  are distinct $(r-1)$-simplices in $W_{2r}(a,x)$ since applying the same proposition to the simplex $\iota_a\op x^{\op 2r}$ shows that they have distinct sets of vertices. In particular, they define distinct morphisms from $x^{\op r}$ to $a\op x^{\op 2r}$ in $\Ccal$, as required. 
    \end{proof}

    \begin{remark}[Sharpness of the connectivity bounds]\label{rem:ranges}
Just like Theorem~\ref{thm:kdouble}, the above result is  not sharp when $r=1$, since it yields a connectivity $\frac{n+k-c}{2k}$ for both 
  $S_n(a,x)$ and $W_n(a,x)$, when we know that both are  $\frac{n-c}{k}$-connected for all $n$. 
The result is in fact also not as good as one would expect for $r\ge 2$ when we look at the resulting homological stability results. Indeed, under our assumption on $\Ccal$, if $W_n(a,x)$ is  $\frac{n-c}{k}$--connected for each $n$, i.e.~slope $\frac{1}{k}$-connected, then Theorem~A in that paper implies that the $r$-stabilization map
      $$\s_r: H_i(\Aut(a\op x^{\op nr}))\to H_i(\Aut(a\op x^{\op nr+r}))$$ is an isomorphism for $i\le \min(\frac{1}{k},\frac{1}{2})nr+b$ for some constant $b$. On the other hand, the complexes $W_n(a,x^{\op r})$ is $\big\lfloor\frac{nr+k(1-2r)-c}{k(r+1)}\big\rfloor$--connected by Proposition~\ref{prop:Sr}. This gives a slope $\frac{r}{k(r+1)}$ in $n$,  
      and applying Theorem~A of \cite{RW-Wahl} to these complexes, we can deduce stability for the same maps $\s_r$ with an isomorphism range of the form
      $i\le \min(\frac{r}{k(r+1)},\frac{1}{2})n+b'$, which is a strictly worse slope when $\frac{1}{k}\le \frac{1}{2}$.

  In the case of symmetric groups, considered as automorphism group in the category $\Ccal=\textrm{FI}$ of finite sets and injections as in  Example~\ref{ex:FI}, we actually know that all the connectivity results playing a role here are sharp. Indeed,  $W_n(\emptyset,[1])$  in that case is the complex of injective words, which is known to be wCM of dimension $n-1$, a connectivity that cannot be improved, see e.g.~\cite[p 613]{farmer}. (But note that its connectivity slope is $1$, which is bigger than $\frac{1}{2}$!)
  The connectivity of $W_n(\emptyset,[2])$ obtained through Proposition~\ref{prop:Sr} is also known to be sharp. Indeed, the connectivity of $\double{S_{2n}(\emptyset,[1])}$ given in Theorem~\ref{thm:Delta} is sharp by \cite[Thm~1.3,1.6]{ShaWac}. This implies that the same holds for  $S_n(\emptyset,[2])$, because the projection to  $\double{S_{2n}(\emptyset,[1])}$ exhibiting it as a join complex (Proposition~\ref{prop:Sjoin}) admits a splitting. Finally, there is likewise a splitting of the forgetful map  $W_n(\emptyset,[2])\to S_n(\emptyset,[2])$, choosing an ordering of the vertices. Hence the  connectivity slope $\frac{r}{r+1}=\frac{2}{3}$ for $W_n(\emptyset,[2])$ cannot be improved. As $\frac{2}{3}>\frac{1}{2}$, the stability result in this case would still give a stability bound of slope $\frac{1}{2}$, even if we use the double, and that is in fact known to be the best possible homological stability slope for symmetric groups. More generally, if one knew that the connectivity of $W_n(\emptyset,[r])$ given in Theorem~\ref{thm:Delta} was likewise best possible, that would still give the correct stability slope in this case for any $r$ because 
  $\min(\frac{r}{r+1},\frac{1}{2})=\frac{1}{2}$  for any~$r$. Of course this works precisely because the original destabilization complex $W_n(a,x)$ was slope~$1$ connected. 
 \end{remark}

\begin{remark}[Monoid cells point of view]\label{rem:cells}
 Following \cite[Sec 7]{krannich}, to construct  the destabilization complex $W_n(a,x)$ it is enough to have a monoidal map from the monoid of braid groups $\mathcal B = \coprod_{n\ge 0}B_n$ to $\mathcal G=\coprod_{n\ge 0} G_n$, if the latter is a monoidal groupoid. This endows $B\mathcal{G}$ with the structure of $E_1$-module over the $E_2$-algebra on $B\mathcal{B}$.  In \cite[Rem 3.10]{MPPRW}, the authors consider the  effect on the destabilization complex of reversing the braiding in $\mathcal G$.  One can interpret the reversed complex as coming from precomposing the map $\mathcal B\to \mathcal G$ by the automorphism of $\mathcal B$ flipping the braid. The resulting complex is homeomorphic to the original one, via a homeomorphism that reverses the order of the face maps in the semi-simplicial sets. 
In this language, the construction presented here comes instead from pre-composing the map $\mathcal B\to \mathcal G$ with a self-map of the braided category 
that is multiplication by $r$ on the objects and takes $r$-fold copies of the strings in each braid. 
This was pointed out to us by Oscar Randal-Williams. 
\end{remark}

\bibliographystyle{plain}
\bibliography{LSW-bibliography}

\end{document}